\documentclass[reqno]{amsart}

\input xy
\xyoption{all}
\usepackage{epsfig}
\usepackage{color}
\usepackage{amsthm}
\usepackage{amssymb}
\usepackage{amsmath}
\usepackage{amscd}
\usepackage{amsopn}
\usepackage{url}
\usepackage{hyperref}\hypersetup{colorlinks}

\usepackage{color} 

\definecolor{darkred}{rgb}{1,0,0} 
\definecolor{darkgreen}{rgb}{0,0.8,0}
\definecolor{darkblue}{rgb}{0,0,1}

\hypersetup{colorlinks,
linkcolor=darkblue,
filecolor=darkgreen,
urlcolor=darkred,
citecolor=darkgreen}

\newcommand{\labell}[1] {\label{#1}}

\numberwithin{equation}{section}
\newtheorem {Theorem}{Theorem}
\numberwithin{Theorem}{section}

\newtheorem {Lemma}[Theorem]    {Lemma}

\newtheorem {Corollary}[Theorem]{Corollary}
\theoremstyle{definition}
\newtheorem{Definition}[Theorem]{Definition}
\theoremstyle{remark}
\newtheorem{Remark}[Theorem]{Remark}
\newtheorem{Example}[Theorem]{Example}

\expandafter\chardef\csname pre amssym.def at\endcsname=\the\catcode`\@
\catcode`\@=11
\def\undefine#1{\let#1\undefined}
\def\newsymbol#1#2#3#4#5{\let\next@\relax
 \ifnum#2=\@ne\let\next@\msafam@\else
 \ifnum#2=\tw@\let\next@\msbfam@\fi\fi
 \mathchardef#1="#3\next@#4#5}
\def\mathhexbox@#1#2#3{\relax
 \ifmmode\mathpalette{}{\m@th\mathchar"#1#2#3}%
 \else\leavevmode\hbox{$\m@th\mathchar"#1#2#3$}\fi}
\def\hexnumber@#1{\ifcase#1 0\or 1\or 2\or 3\or 4\or 5\or 6\or 7\or 8\or
 9\or A\or B\or C\or D\or E\or F\fi}

\font\teneufm=eufm10
\font\seveneufm=eufm7
\font\fiveeufm=eufm5
\newfam\eufmfam
\textfont\eufmfam=\teneufm
\scriptfont\eufmfam=\seveneufm
\scriptscriptfont\eufmfam=\fiveeufm

\catcode`\@=\csname pre amssym.def at\endcsname

\newcommand{\CA}{{\mathcal A}}

\newcommand{\CN}{{\mathcal N}}

\newcommand{\OO}{{\mathcal O}}

\newcommand{\CS}{{\mathcal S}}

\newcommand{\id}{{\mathit id}}

\newcommand{\const}{{\mathit const}}

\newcommand{\fs}{{\mathfrak s}}

\newcommand{\tal}{\tilde{\alpha}}
\newcommand{\tx}{\tilde{x}}

\newcommand{\tF}{\tilde{F}}

\newcommand{\ka}{{\kappa}}

\def    \C      {{\mathbb C}}
\def    \R      {{\mathbb R}}

\def    \Z      {{\mathbb Z}}

\def    \Q      {{\mathbb Q}}

\def    \CP     {{\mathbb C}{\mathbb P}}
\def    \RP     {{\mathbb R}{\mathbb P}}

\def    \12    {{\frac{1}{2}}}

\def    \p      {\partial}

\def    \HF     {\operatorname{HF}}
\def    \HC     {\operatorname{HC}}
\def    \CC     {\operatorname{CC}}

\def    \Gr     {\operatorname{Gr}}

\def    \Fix     {\operatorname{Fix}}

\def    \lcm     {\operatorname{lcm}}

\def    \MUCZ  {\operatorname{\mu_{\scriptscriptstyle{CZ}}}}

\begin{document}


\setlength{\smallskipamount}{6pt}
\setlength{\medskipamount}{10pt}
\setlength{\bigskipamount}{16pt}





\title[Iterated Index and the Mean Euler Characteristic]{Iterated
  Index and the Mean Euler Characteristic}

\author[Viktor Ginzburg]{Viktor L. Ginzburg}
\author[Yusuf G\"oren]{Yusuf G\"oren}

\address{Department of Mathematics, UC Santa Cruz, Santa Cruz, CA
  95064, USA} \email{ginzburg@ucsc.edu} \email{ygoren@ucsc.edu}


\date{\today} 

\thanks{The work is partially supported by NSF grants DMS-1007149 and DMS-1308501.}

\bigskip

\begin{abstract} 
  The aim of the paper is three-fold. We begin by proving a formula,
  both global and local versions, relating the number of periodic
  orbits of an iterated map and the Lefschetz numbers, or indices in
  the local case, of its iterations. This formula is then used to
  express the mean Euler characteristic (MEC) of a contact manifold in
  terms of local, purely topological, invariants of closed Reeb
  orbits, without any non-degeneracy assumption on the
  orbits. Finally, turning to applications of the local MEC formula to
  dynamics, we use it to reprove a theorem asserting the existence of
  at least two closed Reeb orbits on the standard contact $S^3$ (by
  Cristofaro-Gardiner and Hutchings in the most general form) and the
  existence of at least two closed geodesics for a Finsler metric on
  $S^2$ (Bangert and Long).

\end{abstract}

\maketitle

\tableofcontents


\section{Introduction}
\labell{sec:main-results}

In this paper we prove an elementary formula relating the number of
periodic orbits of an iterated map and the Lefschetz numbers of the
iterations. We also establish a local version of this formula which
connects the number of periodic orbits, suitably defined, of a germ at
an isolated fixed point and the indices of its iterations. The latter
result is then used to express the mean Euler characteristic of a
contact manifold with finitely many simple closed Reeb orbits in terms
of local, purely topological, invariants of closed Reeb orbits, when
the orbits are not required to be non-degenerate. Finally, this
expression is utilized to reprove a theorem asserting the existence of
at least two Reeb orbits on the standard $S^3$ (see
\cite{CGH,GHHM,LL}) and the existence of at least two closed geodesics
for a Finsler metric, not necessarily symmetric, on $S^2$ (see
\cite{BL,CGH}).

Let us now describe the results in some more detail.

Consider first a smooth map $F\colon M\to M$, where
$M$ is a closed manifold. We show that the number $I_\ka(F)$ of $\ka$-periodic
orbits of $F$, once the orbits of a certain type are discounted, can be
expressed via the Lefschetz numbers of the iterations $F^d$ for $d|\ka$
(see Theorem \ref{thm:index-maps}), and hence $I_\ka(F)$ is homotopy
invariant. In a similar vein, given a germ of a smooth map at a
fixed point $x$, isolated for all iterations, one can associate to it
a certain invariant $I_\ka(F)$, an iterated index, which counts
$\ka$-periodic orbits of a small perturbation of $F$ near $x$, with
again some orbits being discounted. The iterated index can also be expressed
via the indices of the iterations $F^d$ for $d|\ka$ 
(see Theorem \ref{thm:index-germ}), and hence $I_\ka(F)$ is again a homotopy
invariant of $F$ as long as $x$ remains uniformly isolated for $F^\ka$. We further
investigate the properties of the iterated index in Section
\ref{sec:it-index}, which is essentially independent of the rest of the
paper. The results in this section, although rather elementary, are new
to the best of our knowledge; see however \cite{CMPY}. 

Next, we apply the iterated index to calculation of the mean Euler
characteristic (MEC) of contact manifolds. The MEC of a contact
manifold, an invariant introduced in \cite{VK} (cf.\ \cite{Ra}), is
the mean alternating sum of the dimensions of contact homology. It was
observed in \cite{GK} that when a Reeb flow has finitely many simple
periodic orbits and these orbits are totally non-degenerate, the MEC
can be expressed via certain local invariants of the closed orbits,
computable in terms of the linearized flow. (Here an orbit is called 
totally non-degenerate if all its iterations are non-degenerate.) This
expression for the MEC generalizes a resonance relation for the mean
indices of the closed Reeb orbits on the sphere, proved in
\cite{Vi:res}, and is further generalized to certain cases
where there are infinitely many orbits in \cite{Es}, including the
Morse-Bott setting. The results of this type have been used for
calculations of the MEC (see, e.g., \cite{Es}) and also in
applications to dynamics; \cite{GK,GG:generic,Ra}. It is this latter
aspect of the MEC formula that we are interested in here.

Versions of the MEC formula where the Reeb orbits are allowed to
degenerate are established in \cite{HM,LLW} with applications to
dynamics in mind. In these formulas, however, the contributions of
closed orbits are expressed in terms of certain local homology groups
associated with the orbits and are contact-geometrical in
nature. Here, in Theorem \ref{thm:mec}, we extend the result from
\cite{GK} to the degenerate case with the orbit contributions
computable via the linearized flow and a certain purely topological
invariant of the Poincar\'e return map. The latter invariant is
essentially the mean iterated index.

Two remarks concerning this result are due now.

First of all, it should be noted that the authors are not aware of any
example of a Reeb flow with finitely many closed Reeb orbits such that
not all of them are totally non-degenerate. (It is conceivable
that such Reeb flows do not exist, although proving this is likely to be
absolutely out reach at the moment.) Thus the degenerate case
of the local MEC formula (Theorem \ref{thm:mec}) appears to be of
little interest for the MEC calculations for specific
manifolds. However, it does have applications to dynamics. For
instance, it allows one to rule out certain orbit patterns and, as
a consequence, obtain lower bounds on the number of closed orbits (cf.\
\cite{HM,LL,WHL}); see Section \ref{sec:S^3} and a discussion below.

Secondly, our local MEC formula is identical to the one established in
\cite{HM} and also, apparently, to the one proved for the sphere in
\cite{LLW}. The difference lies in the interpretation or the
definition of the terms in the formula. Although our local MEC formula
could be directly derived from \cite[Theorem 1.5]{HM} (see Remark
\ref{rmk:chi}), we give, for the sake of completeness, a
rather short proof of the formula, still relying, however, on some of the
results from \cite{HM}; cf.\ \cite{GHHM}. In Section \ref{sec:mec}, we
also discuss in detail the definition of the MEC, examples, the local
and filtered contact homology and other ingredients of the proof of
Theorem \ref{thm:mec}, and state a variant of the asymptotic Morse
inequalities for contact homology (Theorem \ref{thm:morse})
generalizing some results from \cite{GK,HM}.
 
Finally, in the last section, we turn to applications of Theorem
\ref{thm:mec}. We give a new proof of the theorem that every Reeb flow
on the standard contact $S^3$ has at least two closed Reeb orbits; see
\cite{CGH,GHHM,LL}.  (In fact, a more general result is now known to
be true. Namely, the assertion holds for any contact
three-manifold. This fact, proved in \cite{CGH} using the machinery of
embedded contact homology, is out of reach of our methods.)  The
proofs in \cite{GHHM,LL} both rely on an analogue of the ``degenerate
case of the Conley conjecture'' for contact forms, asserting that the
presence of one closed Reeb orbit of a particular type (a so-called
symplectically degenerate minimum) implies the existence of infinitely
many closed Reeb orbits.  This result holds in all dimensions; see
\cite{GHHM}. Another non-trivial (and strictly three-dimensional)
ingredient in the argument in \cite{GHHM} comes from the theory of
finite energy foliations (see \cite{HWZ1,HWZ2}), while the argument in
\cite{LL} uses, also in a non-trivial way, the variant of the local
MEC formula from \cite{LLW} for degenerate Reeb flows on the standard
contact sphere. In this paper, we bypass the results from the theory
of finite energy foliations and give a very simple proof of Theorem
\ref{thm:S^3} utilizing Theorem \ref{thm:mec} and the ``Conley
conjecture'' type result mentioned above. The advantage of this
approach is that it minimizes the 3-dimensional counterparts of the
proof and, we hope, is the first step towards higher-dimensional
results. We also reprove, in slightly more general form, a theorem
from \cite{BL} that every Finsler metric on $S^2$ has at least two
closed geodesics. (Clearly, this also follows from \cite{CGH}.)

A word is due on the degree of rigor in this paper, which varies
considerably between its different parts. Section \ref{sec:it-index},
dealing with the iterated index, is of course completely rigorous.
The rest of the paper, just as \cite{GHHM}, heavily relies on the
machinery of contact homology (see, e.g., \cite{Bo,SFT} and references
therein), which is yet to be fully put on a rigorous basis (see
\cite{HWZ:SC,HWZ:poly}). Note however that, to get around the
foundational difficulties, one can replace here, following, e.g.,
\cite{CDvK,FSVK}, the linearized contact homology by the equivariant
symplectic homology, which carries essentially the same information (see
\cite{BO12}), at the expense of proofs getting somewhat more involved;
cf.\ Remarks \ref{rmk:cyl} and \ref{rmk:var}, and \cite{HM} vs.\
\cite{McL} or \cite{Es} vs.\ \cite{FSVK}.

\medskip
\noindent\textbf{Acknowledgements.} The authors are grateful to Alexander
Felshtyn, John Franks, Ba\c sak G\"urel, Umberto Hryniewicz, Janko
Latschev, Leonardo Macarini, John Milnor, and Otto van Koert for
useful discussions and remarks.

\section{Iterated index}
\label{sec:it-index}
In this section, we establish a few elementary results concerning the
count, with multiplicity and signs, of periodic orbits of smooth
maps. Although we feel that these results must be known in some form,
we are not aware of any reference; see however \cite{CMPY,Do}, and
also \cite[Chap.\ 3]{JM} and \cite[Sect.\ I.4]{Sm}, for related
arguments. Throughout the section, all maps are assumed to be at least
$C^1$-smooth unless explicitly stated otherwise.

\subsection{Iterated index of a map}
As a model situation, consider a $C^1$-smooth map $F\colon M\to M$, where
$M$ is a closed manifold. We are interested in an algebraic count of
periodic orbits of $F$. To state our main results, let us first review
some standard definitions and facts.

Denote by $\Fix(F)$ the set of the fixed points of $F$. Recall that
$x$ is a $\ka$-periodic point of $F$ if $F^\ka(x)=x$, i.e., $x\in
\Fix(F^\ka)$, and that $\tau$ is the \emph{minimal period} of $x$ if $\tau$
is the smallest positive integer such that $F^\tau(x)=x$. The
$\ka$-periodic orbit containing $x$ is the collection $\OO=\{x, F(x),
\ldots, F^{\ka-1}(x)\}$ of not necessarily distinct points (naturally
parametrized by $\Z/\ka\Z$). The minimal period is then the length
(i.e., the cardinality) of the orbit or, more precisely, of its image
in $M$. Note that necessarily~$\tau|\ka$.

The \emph{index} $I(F,x)$ of an isolated fixed point $x$ of $F$ is the
degree of the map $S^{n-1}\to S^{n-1}$ given, in a local chart
containing $x$, by
$$
z\mapsto \frac{z-F(z)}{\|z-F(z)\|}
$$
where $z$ belongs to a small sphere $S^{n-1}$ centered at $x$. It is
not hard to see that all points in a $\ka$-periodic orbit $\OO$ have
the same index as fixed points of $F^\ka$. Thus, setting
$I(F^\ka,\OO):=I(F^\ka,x)$ for any $x\in\OO$, we have the index
assigned to a $\ka$-periodic orbit. 
 
Furthermore, recall that a fixed point $x$ of $F$ is said to be
\emph{non-degenerate} if 1 is not an eigenvalue of the linearization
$DF_x\colon T_xM\to T_xM$ and that $F$ is called non-degenerate if all
its fixed points are non-degenerate.  To proceed, let us first assume
that all periodic points of $F$ are non-degenerate, i.e., all
iterations $F^\ka$ are non-degenerate. (As follows from (a part of)
the Kupka--Smale theorem, this is a $C^\infty$-generic condition; see,
e.g., \cite{AR,Sm}.)  Then the index $I(F^\ka,x)$ is equal to
$(-1)^m$, where $m$ is the number of real eigenvalues of $DF_x^\ka$ in
the range $(1,\, \infty)$.

\begin{Definition}
  \label{def:main} Let $x$ be a periodic point of $F$ with minimal
  period $\tau$. We say that $x$ is \emph{even (odd)} is the number of
  real eigenvalues of $DF^\tau_x \colon T_xM\to T_xM$ in the interval
  $(-\infty,\, -1)$ is even (odd). A $\ka$-periodic point $x$ with
  minimal period $\tau$ is said to be \emph{bad} if it is an even
  iteration of an odd point, i.e., $x$ is odd and the ratio $\ka/\tau$ is
  even. Otherwise, $x$ is said to be \emph{good}. A $\ka$-periodic
  orbit $\OO$ is bad (good) if one, or equivalently all, periodic points in
  $\OO$ are bad (good).
\end{Definition}

Alternatively, the difference between even and odd (or good and bad)
periodic points can be seen as follows. Let $x$ be a periodic point
with minimal period $\tau$. We can also view $x$ as a $\ka$-periodic
point for any positive integer $\ka$ divisible by $\tau$. Then
$I(F^\ka,x)=I(F^\tau,x)$ when $x$ is even and
$I(F^\ka,x)=(-1)^{1+\ka/\tau}I(F^\tau,x)$ when $x$ is odd. Finally, $x$,
viewed as a $\ka$-periodic point, is good or bad depending on whether
$I(F^\ka,x)=I(F^\tau,x)$ or not. Since all periodic points in a periodic
orbit have the same index, this definition extends to periodic orbits.

The terminology we use here is borrowed from the theory of contact
homology (see, e.g., \cite{Bo,SFT}). In dynamics, odd periodic orbits
are sometimes also referred to, at least for flows, as M\"obius orbits
(see \cite{CMPY}) or flip orbits. Furthermore, note that
the above discussion relies heavily on the fact that no root of unity
is an eigenvalue of $DF_x^k$ due to the non-degeneracy assumption.

As is well known, the Lefschetz number $I(F):=\sum_{x\in\Fix(F)}
I(F,x)$ is a homotopy invariant of $F$; see, e.g., \cite{Fr,JM}. (In
particular, $I(F)$ can be extended ``by continuity'' to all, not
necessarily non-degenerate, maps $F$.) This, of course, applies to
$F^\ka$ as well, and hence $I(F^\ka) = \sum_{x\in\Fix(F^\ka)}
I(F^\ka,x)$, the number of periodic points counted with signs, is also
homotopy invariant. However, the number of $\ka$-periodic orbits,
taken again with signs, is not homotopy invariant, i.e.,
$\sum_{\OO}I(F^\ka,\OO)$, where the summation extends to all (not
necessarily simple) $\ka$-periodic orbits, can vary under a
deformation of $F$. (An example is, for instance, the second iteration
of the period doubling bifurcation map in dimension one, starting with
an attracting fixed point with an eigenvalue in $(-1,\,0)$ and ending
with an odd repelling fixed point and an (even) attracting orbit of
period two, with the sum for $\ka=2$ changing from 1 before the
bifurcation to 0 afterwards; cf.\ \cite{YA}.)  However, this sum is
very close to being homotopy invariant and it becomes such once the
summation is restricted to good orbits only. Namely, assuming as above
that all periodic points of $F$ are non-degenerate, set
$$
I_\ka(F):=\sum_{\text{ good } \OO}I(F^\ka,\OO),
$$
where the sum is now taken over all good $\ka$-periodic orbits $\OO$
of $F$, not necessarily with minimal period $\ka$. We call $I_\ka(F)$
the \emph{iterated index} of $F$. In the example
of a period doubling bifurcation mentioned above, we have $I_2(F)=1$
before and after the bifurcation. We emphasize that $I_\ka(F)$ depends
in general not only on $F^\ka$, but separately on $\ka$ and $F$.

Let $\varphi$ be the Euler function, i.e., $\varphi(\ka)$ is the
number of positive integers which are smaller than $\ka$ and relatively
prime with $\ka$. By definition, $\varphi(1)=1$.

\begin{Theorem}
\label{thm:index-maps}
Assume, as above, that $F^\ka$ is non-degenerate. Then
\begin{equation}
\label{eq:index-maps}
I_\ka(F)=\frac{1}{\ka}\sum_{d|\ka} \varphi(\ka/d)I(F^d).
\end{equation}
\end{Theorem}

Since the right hand side of \eqref{eq:index-maps} is obviously
homotopy invariant, so is $I_\ka(F)$. Furthermore, it extends by
continuity to all $F$, not necessarily meeting the non-degeneracy
requirement. Namely, let $\tilde{F}$ be a small perturbation of $F$
such that all $\ka$-periodic points of $\tilde{F}$ are
non-degenerate. Then $I_\ka(F):=I_\ka(\tilde{F})$ is well defined,
  i.e., independent of the perturbation $\tilde{F}$.

\begin{Example}
  Assume that $F$ is homotopic to the identity. Then, so is $F^d$ for
  all $d$, and, by Euler's formula (see \eqref{eq:euler} below), we
  have $I(F^\ka)=I(F)=I_\ka(F)$ for all $\ka\geq 1$.
\end{Example}

\begin{proof}[Proof of Theorem \ref{thm:index-maps}]
  A $\ka$-periodic orbit $\OO$ with minimal period $\tau=|\OO|$
  contributes to $I_\ka(F)$ only when $\tau|\ka$, and to the
  individual terms on the right hand side of \eqref{eq:index-maps}
  when $\tau|d$. Set $\ka=\tau\ka'$ and $d=\tau d'$.

When $\OO$ is good, its contribution to $I_\ka(F)$ is equal to
$I(F^\ka,\OO)=I(F^\tau,\OO)$. On the other hand, its contribution to the right hand
side of \eqref{eq:index-maps} is
$$
\frac{1}{\ka}\sum_{\tau|d|\ka} \varphi(\ka/d)I(F^d,\OO)\tau
=\frac{1}{\ka'}\sum_{d'|\ka'}\varphi(\ka'/d')I(F^\tau,\OO)=I(F^\tau,\OO),
$$
where we use Euler's formula
\begin{equation}
\label{eq:euler}
\sum_{l|r}\varphi(l)=r;
\end{equation}
see, e.g., \cite[Theorem 63]{HW}.

Assume next that $\OO$ is bad. Then $\ka'=\ka/\tau$ is even and $\OO$,
viewed as a $\tau$-periodic orbit, is odd, and
$I(F^d,\OO)=(-1)^{1+d/\tau}I(F^\tau,\OO)$. Thus the contribution of $\OO$ to
$I_\ka(F)$ is zero and its contribution to the right hand side is
$$
\frac{1}{\ka}\sum_{\tau|d|\ka} \varphi(\ka/d)I(F^d,\OO)\tau
=-\frac{1}{\ka'}\sum_{d'|\ka'}\varphi(\ka'/d')(-1)^{d'}I(F^\tau,\OO)=0,
$$  
where we now use the fact that
\begin{equation}
\label{eq:identity}
\sum_{l|r}(-1)^{l}\varphi(r/l)=0
\end{equation}
for any positive even integer $r$.

Finally, \eqref{eq:identity} can be proved, for instance, as
follows. Set $r=2^mc$, where $c$ is odd and $m\geq 1$ since $r$ is
even. Then, regrouping the terms on the left hand side of
\eqref{eq:identity}, we have
$$
\sum_{l|r}(-1)^{l}\varphi(r/l)=\sum_{a|c} P(a),
$$
where
$$ 
P(a)=
(-1)^{2^ma}\varphi(c/a)+(-1)^{2^{m-1}a}\varphi(2c/a)+\ldots+(-1)^a\varphi(2^m
c/a),
$$
and it suffices to show that $P(a)=0$.
Set $b=c/a$ and recall that $\varphi(2^jb)=2^{j-1}\varphi(b)$
since $\varphi$ is multiplicative and $b$ is odd; \cite{HW}. When $m=1$, we clearly
have $P(a)=\varphi(b)-\varphi(b)=0$. Likewise, as is easy to see,
$$
P(a)=\varphi(b)+\varphi(b)+\ldots+2^{m-2}\varphi(b)-2^{m-1}\varphi(b)=0
$$
when $m\geq 2$. This completes the proof of the theorem.
\end{proof}

\begin{Remark}
\label{rmk:C^0-maps}
The fact that the iterated index is homotopy invariant enables one to
extend the definition of $I_\ka(F)$ to continuous maps $F\colon M\to
M$, when $M$ is still a smooth manifold, by setting
$I_\ka(F)=I_\ka(\tF)$ where $\tF$ is a smooth approximation of
$F$. Clearly, Theorem \ref{thm:index-maps} still holds in this case.
\end{Remark}

\subsection{Iterated index of a germ}
\label{sec:index-germ}
Consider now a germ of a $C^1$-smooth map $F$ at a
fixed point $x$, which is assumed to be isolated for all iterations
$F^\ka$ but not necessarily non-degenerate. Thus the index
$I(F^\ka,x)$ is defined and homotopy invariant as long as $x$
remains \emph{uniformly isolated}: $I(F^\ka_s,x)=\const$ when
$F_s$ varies smoothly or continuously with parameter $s$ and there
exists a neighborhood $U$ of $x$ such that $x$ is the only fixed point
of $F^\ka_s$ in $U$ for all $s$.  

For a given $\ka$, consider a sufficiently small
perturbation $\tilde{F}$ such that all $\ka$-periodic points of
$\tilde{F}$ are non-degenerate and set 
$$
I_\ka(F,x):=\sum_{\text{ good } \OO}I(\tilde{F}^\ka,\OO),
$$
where the sum is again taken over all good $\ka$-periodic orbits $\OO$ of
$\tilde{F}$. For instance, when $F$ and all its iterations are
non-degenerate, we have $I_\ka(F,x)=I(F,x)$ when $x$ is even or when
$x$ is odd and $\ka$ is odd, and $I_\ka(F,x)=0$ otherwise. In what follows, when the point $x$
is clear from the context, we will use the notation $I(F)$ and $I_\ka(F)$.

\begin{Theorem}
\label{thm:index-germ}
We have
\begin{equation}
\label{eq:index-germ}
I_\ka(F)=\frac{1}{\ka}\sum_{d|\ka} \varphi(\ka/d)I(F^d).
\end{equation}

\end{Theorem}

We omit the proof of this theorem; for it is word-for-word
identical to the proof of Theorem \ref{thm:index-maps}. An
immediate consequence is

\begin{Corollary}
\label{cor:index-germ}
The iterated index $I_\ka(F,x)$ is well defined and homotopy invariant
as long as $x$ is a uniformly isolated fixed point of $F^\ka_s$. 
\end{Corollary}

\begin{Remark}
\label{rmk:C^0-germs} As in the case of maps of smooth manifolds (see Remark
\ref{rmk:C^0-maps}), the definition of the iterated index $I_\ka(F,x)$
extends to continuous maps $F$ and Theorem \ref{thm:index-germ} remains
valid. In what follows, however, the requirement that $F$ is at least
$C^1$ becomes essential; cf.\ \cite{SS}.
\end{Remark}

Assuming that $x$ is isolated for all iterations $F^\ka$, consider the
sequence of indices $\iota_\ka=I(F^\ka)$. This sequence is bounded
(see \cite{SS}) and in fact periodic; see \cite{CMPY} and also, e.g.,
\cite{JM} and references therein. (It is essential here that $F$ is at
least $C^1$-smooth.) Moreover, there exists a finite collection of
positive integers $\CN$ with the following properties:
\begin{itemize}
\item[(i)] $1\in\CN$;
\item[(ii)] for any two elements $q$ and $q'$ in $\CN$, the least common
  multiple $\lcm(q,q')$ is also an element of $\CN$;
\item[(iii)] for any $\ka$, we have $\iota_\ka=\iota_{q(\ka)}$, where $q(\ka)$
    is the largest element in $\CN$ dividing~$\ka$. 
\end{itemize}

Thus the index sequence $\iota_\ka$ is $q_{\max}$-periodic, where
$q_{\max}=\max\CN$, and takes values in the set $\{ \iota_q\mid
q\in\CN\}$. The collection $\CN=\CN(F)$ is generated in the obvious sense by
the degrees of the roots of unity occurring among the eigenvalues of
the linearization $DF_x$ and, in addition, $2\in \CN$ when $x$ is
odd. (For instance, if all eigenvalues of $DF_x$ are equal to $1$, we
have $\CN=\{1\}$ and the sequence $\iota_\ka$ is constant; when none
of the eigenvalues is a root of unity, i.e., $x$ is non-degenerate for
all iterations, we have either $\CN=\{1\}$ or $\CN=\{1,\, 2\}$ depending on
whether $x$ is even or odd.)

Condition (iii) relating the sequence and the set $\CN$ satisfying
(i) and (ii) is particularly important, and we say that a sequence
$a_\ka$ is \emph{subordinated} to $\CN$ when (iii) holds: for any
$\ka$, we have $a_\ka=a_{q(\ka)}$, where $q(\ka)$ is the largest
element in $\CN$ dividing $\ka$. (Note that (iii)
is an easy consequence of the fact, implicitly contained in the proof of the
Shub--Sullivan theorem \cite{SS}, that $I(F^\ka)=I(F)$ wherever $\ka$
is relatively prime with all $q\in\CN$; cf.\ \cite{CMPY,GG:gap}.)

\begin{Theorem}
\label{thm:subordinate}
The sequence $I_\ka(F)$ is subordinated to $\CN$.
\end{Theorem}

In particular, it follows that $I_\ka(F)$ is also periodic with
period $q_{\max}$.

\begin{proof}
The result is a formal consequence of Theorem
\ref{thm:index-germ}. Namely, let $a_\ka$ be any sequence subordinated
to a finite set $\CN$ satisfying (i) and (ii), and let $b_\ka$ be
obtained from $a_\ka$ via \eqref{eq:index-germ}:
$$
b_\ka=\frac{1}{\ka}\sum_{d|\ka} \varphi(\ka/d)a_d.
$$
Then we claim that $b_\ka$ is also subordinated to $\CN$.

At this point the fact that the original sequences are integer-valued
becomes inessential, and it is more convenient to assume that the
sequences in question are real. We can view the transform
$\{a_\ka\}\mapsto \{b_\ka\}$ as a map $\Phi$ from the vector space $\R^\CN$
of all real sequences subordinated to $\CN$ to the vector space of all
sequences. Our goal is to show that $\Phi$ actually sends $\R^\CN$ to
itself.

Consider the basis $\delta(q)$, $q\in\CN$, of $\R^\CN$, where
$\delta(q)_\ka=1$ when $q|\ka$ and $\delta(q)_\ka=0$ otherwise. Then,
as we will prove shortly, $\Phi$ is diagonal in this basis with entries
$1/q$, i.e., $\Phi\big(\delta(q)\big)=\delta(q)/q$. In particular, $\Phi\colon
\R^\CN\to\R^\CN$ and the theorem follows.

By definition,
$$
\Phi\big(\delta(q)\big)_\ka =\frac{1}{\ka}\sum_{d|\ka}
\varphi(\ka/d)\delta(q)_d.
$$
The only non-zero terms on the right hand side are those with $q|d$.
Thus, when $\ka$ is not divisible by $q$, the right hand side
is zero and hence $\Phi\big(\delta(q)\big)_\ka=0$. When $q|\ka$, set
$\ka=q\ka'$ and $d=q d'$ as in the proof of Theorem \ref{thm:index-maps}. Then
$$
\Phi\big(\delta(q)\big)_\ka =\frac{1}{q\ka'}\sum_{d'|\ka'}
\varphi(\ka'/d')=\frac{1}{q},
$$
where we again used \eqref{eq:euler}. This completes the proof of the theorem.
\end{proof}

\begin{Remark}[Integrality]
  In addition to being subordinated to $\CN$, the sequence
  $\iota_\ka=I(F^\ka)$ is known to satisfy some further
  conditions. Namely, \cite[Theorem 2.2]{CMPY} asserts that, in the
  notation from the proof of Theorem \ref{thm:subordinate}, the
  sequence $\iota_\ka$ is an integral linear combination of the
  sequences $q\delta(q)$, where $q\in\CN$. (Furthermore, the
  coefficients must meet certain ``sign-reversing'' constraints if $x$
  is odd.)  Moreover, it follows that
  the mean index (the $\varphi$-index in the terminology of
  \cite{CMPY}) is an integer:
$$
\bar{\iota}=\lim_{N\to\infty}\frac{1}{N}\sum_{\ka=1}^N\iota_\ka
=\frac{1}{q_{\max}}\sum_{\ka=1}^{q_{\max}}\iota_\ka\in\Z;
$$
see \cite[Corollary 2.3]{CMPY}.

Arguing as in the proof of Theorem \ref{thm:subordinate}, one can
derive \cite[Theorem 2.2]{CMPY} from the fact that the right hand side
of \eqref{eq:index-germ} is an integer by the definition of $I_\ka(F)$,
and conversely \cite[Theorem 2.2]{CMPY} implies that the right hand
side of \eqref{eq:index-germ} is an integer since $\Phi(q\delta(q))=\delta(q)$.
\end{Remark}

\begin{Example} 
\label{ex:2D}
Assume that $2n=2$ and $F$ is elliptic with eigenvalues $e^{\pm 2\pi
  i\alpha}$, where $\alpha$ is rational: $\alpha=p/q$ with $p$ and $q$
relatively prime and $q\geq 2$. Then it follows from \cite[Theorem
2.2]{CMPY} or our Theorem \ref{thm:subordinate} that $I(F^\ka)=1-rq$
when $q | \ka$ and $I(F^\ka)=1$ otherwise. Thus $I_\ka(F)=1-r$ when $q
| \ka$ and $I_\ka(F)=1$ otherwise. In addition, when $F$ is area
preserving, $I(F^\ka)\leq 1$ (see \cite{Ni,Si}), and hence $r\geq 0$
in this case. There are no other restrictions on the index sequence
when $2n=2$: all values of $r$ and $q$ do occur. For instance, in the
area preserving case, one can take as $F$ the composition of the rotation
$e^{2\pi i/q}$ with the flow of the ``monkey saddle'' Hamiltonian $\Re
(z^{qr})$.  One immediate consequence of this result, otherwise
entirely non-obvious, is that $I(F^\ka)\neq 0$ for all $\ka$, when
$I(F)\neq 0$. (Remarkably, under some additional assumptions, these
facts remain true for homeomorphisms; see \cite{LeC,LeCY} and
references therein.)
\end{Example}

\section{Mean Euler characteristic}
\label{sec:mec}

\subsection{Notation and conventions}
\label{sec:setting}
Let $(M^{2n-1},\xi)$ be a closed contact manifold strongly fillable by
an exact symplectically aspherical manifold. In other words, we require
$(M,\xi)$ to admit a contact form $\alpha$ such that there exists an exact
symplectic manifold $(W,\omega=d\alpha_W)$ with $M=\p W$
(with orientations) and $\alpha_W|_M=\alpha$ and $c_1(TW)=0$. Then
the linearized contact homology $\HC_*(M,\xi)$ is defined and
independent of $\alpha$; see, e.g.,
\cite{Bo,SFT}. Although $\HC_*(M,\xi)$ depends in general on the filling
$(W,\omega)$, we suppress this dependence in the
notation. Moreover, when $M$ is clear from the context, we will
simply write $\HC_*(\xi)$.

When $\alpha$ is non-degenerate, $\HC_*(\xi)$ is the homology of a
complex $\CC_*(\alpha)$ generated over $\Q$ by the good closed, not
necessarily simple, Reeb orbits $x$ of $\alpha$, where an orbit is
said to be good/bad depending on whether the corresponding fixed point
of the Poincar\'e return map is good/bad. The complex is graded by the
Conley--Zehnder index up to a shift of degree by $n-3$, i.e.,
$|x|=\MUCZ(x)+n-3$. The exact nature of the differential on
$\CC_*(\alpha)$ is inessential for our considerations.

The complex $\CC_*(\alpha)$ further breaks down into a direct sum of
sub-complexes $\CC_*(\alpha;\gamma)$ generated by the closed Reeb
orbits in the free homotopy class $\gamma$ of loops in $W$.  When
$\gamma\neq 0$, fixing grading (or, equivalently, a way to evaluate the
Conley--Zehnder index of $x$) requires fixing an extra structure. A
convenient choice of such an extra structure in our setting is a
non-vanishing section $\fs$, taken up to homotopy, of the square of
the complex determinant line bundle $\big(\bigwedge_\C
TW\big)^{\otimes 2}$; see \cite{Es}. (The section $\fs$ exists since
$c_1(TW)=0$.)  With this choice, for every closed Reeb orbit $x$, its
Conley--Zehnder index $\MUCZ(x)$ and the mean index $\Delta(x)$ are
well-defined. Namely, the indices of $x$ are evaluated using a
(unitary) trivialization of $x^*TW$ such that the square of its top
complex wedge is $\fs|_x$. Such a trivialization is unique up to
homotopy.  Moreover, the mean index is homogeneous with respect to the
iteration, i.e., $\Delta(x^\ka)=\ka\Delta(x)$. We refer the reader to
\cite{Es} for a very detailed discussion of the mean index in this
context and for further references, and also to \cite{Lo,SZ}. Here we
only mention that the mean index $\Delta(x)$ measures the total
rotation of certain eigenvalues on the unit circle of the linearized
Poincar\'e return map of the Reeb flow at $x$. In general, the grading
does depend on the choice of $\fs$.

Finally, let us also fix a collection $\Gamma$ of free homotopy
classes in $W$ closed under iterations, i.e., such that $\gamma^\ka\in
\Gamma$ whenever $\gamma\in\Gamma$. (For instance, we can have
$\Gamma=\{0\}$ or $\Gamma$ can be the entire collection of free
homotopy classes.)  In what follows, we focus on the homology
$\bigoplus_\Gamma\HC_*(\xi,\gamma)$. We again suppress the dependence
of the homology on $\fs$ and $\Gamma$ (and on $W$) in the notation and
$\HC_*(\xi)$.

Furthermore, assume that 
\begin{itemize}
\item[(CH)]  there are two integers $l_+$ and $l_-$ such that  the space 
$\HC_l(\xi)$ is finite-dimensional for
$l\geq l_+$ and $l\leq l_-$.  
\end{itemize} 
In all examples considered here the contact homology is finite
dimensional in all degrees and this condition is automatically
met. Set
\begin{equation}
\label{eq:Euler1}
\chi^\pm(\xi)=\lim_{N\to\infty}\frac{1}{N}
\sum_{l=l_\pm}^N(-1)^l\dim \HC_{\pm l}(\xi),
\end{equation}
provided that the limits exist.  We call $\chi^\pm(\xi)$ the
\emph{positive/negative mean Euler characteristic (MEC)} of
$\xi$. (Invariants of this type for contact manifolds are originally
introduced and studied in \cite[Section 11.1.3]{VK}; see also
\cite{Es,GK,Ra}.) Note that $\chi^\pm(\xi)$ depends of course on
$\Gamma$ and also, when $\Gamma\neq \{0\}$, on $\fs$. (However, when,
say, $M$ is simply connected, it is not hard to see that
$\chi^\pm(\xi)$ is independent of the filling under some natural
additional conditions on $\xi$; cf.\ \cite{CDvK,VK:br,OV}.)

Assume now that the dimensions
of the contact homology spaces remain bounded as $l\to\pm\infty$:
\begin{equation}
\label{eq:bound}
\dim_\Q \HC_l(\xi)\leq \const \text{ when } |l|\geq l_\pm.
\end{equation}
Then, although the limit in \eqref{eq:Euler1} need not exist, one can
still define $\chi^\pm(\xi)$ as, e.g., following \cite{CDvK},
$$
\chi^\pm(\xi)=\frac{1}{2}\left[
\limsup_{N\to\infty}\frac{1}{N}
\sum_{l=l_\pm}^N(-1)^l\dim \HC_{\pm l}(\xi)
+
\liminf_{N\to\infty}\frac{1}{N}
\sum_{l=l_\pm}^N(-1)^l\dim \HC_{\pm l}(\xi)
\right].
$$

Finally, let us also point out that the machinery of contact
homology used in this section is yet to be fully put on a rigorous
basis and the foundations of the theory is still a work in progress;
see \cite{HWZ:SC,HWZ:poly}.

\begin{Remark}
\label{rmk:cyl}
One can also use the cylindrical contact homology, when it exists, to
define the MEC of a contact manifold $(M,\xi)$. The construction is
similar to the one for the linearized contact homology with obvious
modifications. For instance, in the cylindrical case, $\fs$ is a
non-vanishing section of the line bundle
$\big(\bigwedge_\C\xi\big)^{\otimes 2}$ and the collection $\Gamma$ is
formed by free homotopy classes of loops in $M$. Of course, the two
definitions give the same result when the cylindrical and contact
homology groups are equal or when the dimension of the contact complex
is bounded as a function of the degree and the MEC can be evaluated
using the complex rather than the homology. This is the case, for
example, in the setting of Theorem \ref{thm:mec}.

Alternatively, the MEC can be defined using the positive equivariant
symplectic homology (see \cite{CDvK,FSVK}), resulting in an invariant
(somewhat hypothetically) equal to the one obtained using the
linearized contact homology; \cite{BO12}. Note that this approach
bypasses the foundational difficulties related to contact homology but
usually results in somewhat more involved proofs and calculations;
cf.\ Remark \ref{rmk:var}.  Finally, variants of the MEC exist and
have been used for ``classical'' homology theories; see, e.g.,
\cite{LLW,Ra,WHL}.

\end{Remark}

\subsection{Examples}
\label{sec:examples}
In this section we briefly review, omitting the proofs, some examples
where the MEC is not difficult to determine. 

\begin{Example}[The Standard Sphere] 
\label{eq:xi_0}
For the standard contact structure $\xi_{0}$ on $S^{2n-1}$, the
contact homology is one-dimensional in every even degree starting with
$2n-2$; see, e.g., \cite{Bo}. Thus, in this case, $\chi^+=1/2$ and
$\chi^-=0$. Alternatively, one can use here the Morse--Bott approach
as in Example \ref{ex:pre-quant}; see \cite{Es}.
\end{Example}

\begin{Example}[The Ustilovsky Spheres] In \cite{U}, Ustilovsky
  considers a family of contact structures $\xi_p$ on $S^{2n-1}$ for
  odd $n$ and positive $p\equiv \pm 1\mod 8$.  For a fixed $n$, the
  contact structures $\xi_p$ fall within a finite number of homotopy
  classes, including the class of the standard structure $\xi_0$.  
The contact homology $\HC_*(\xi_p)$ is computed in \cite{U},
and it is not hard to see that in this case
\begin{equation}
\label{eq:ust}
\chi^+(\xi_p)=\frac{1}{2}
\left(
\frac{p(n-1)+1}{p(n-2)+2}
\right)
\end{equation}
and $\chi^-(\xi_p)=0$; see \cite{VK:br,VK} and also \cite{Es,GK}
for a different approach. The right-hand side of \eqref{eq:ust} is a
strictly increasing function of $p>0$.  Hence, $\chi^+$ distinguishes
the structures $\xi_p$ with $p>0$.  Note also that
$\chi^+(\xi_p)>\chi^+(\xi_0)=1/2$ when $p>1$ and
$\chi^+(\xi_1)=1/2$. In particular, $\chi^+$ distinguishes $\xi_p$
with $p>1$ from the standard structure $\xi_0$.
\end{Example}

\begin{Example}[Pre-quantization Circle Bundles] 
\label{ex:pre-quant}
Let $\pi\colon M^{2n-1}\to B$ be a prequantization circle bundle over
a closed strictly monotone symplectic manifold $(B,\omega)$. In other words,
$M$ is an $S^1$-bundle over $B$ with $c_1=[\omega]$, i.e., we have
$\pi^*\omega = d\alpha$, where $\alpha$ is a connection form on $M$
and we use a suitable identification of the Lie algebra of $S^1$ and
$\R$; see, e.g., \cite[Appendix A]{GGK}. As is well known, $\alpha$ is
a contact form.  Assume first for the sake of simplicity that the
fiber $x$ of $\pi$ is contractible. Then, for $\Gamma=\{0\}$, we have
$\chi^-(\xi)=0$ and
$$
\chi^+(\xi)=\frac{\chi(B)}{2\left<c_1(TB),u\right>},
$$
where $u\in\pi_2(B)$ is the image of a disk bounded by $x$ in $M$ and
$\chi(B)$ is the ordinary Euler characteristic of the base $B$. (Here
we are using the cylindrical contact homology of $(M,\xi)$ rather than
the linearized contact homology because the natural filling of $M$ by
the disk bundle is not exact or even symplectically aspherical.) This
is an easy consequence, for instance, of the Morse--Bott version of
\eqref{eq:mec} and of the fact that, essentially by definition, the
denominator is $\Delta(x)$; see \cite[Example 8.2]{Es}. Furthermore,
since $x$ is contractible, $\left<\omega,u\right>=1$ and $B$ is
monotone, $\left<c_1(TB),u\right>$ is the minimal Chern number $N$ of
$B$; cf. \cite[p.\ 100]{Bo:thesis}. To summarize, we have
\begin{equation}
\label{eq:mec-cot}
\chi^+(\xi)=\frac{\chi(B)}{2N}.
\end{equation}
If $x$ is not contractible, the above calculation still holds for
$\Gamma=\{0\}$. More generally, when the order $r$
of the class $[x]$ in $\pi_1(M)$ is finite, for the collection
$\Gamma$ generated by $[x]$ and any $\fs$, we have
$\chi^+(\xi)=r\chi(B)/(2N)$.  When $B$ is negative monotone, the roles
of $\chi^+$ and $\chi^-$ are interchanged. (See \cite{CDvK} for far
reaching generalizations of this calculation.) 
\end{Example}

\begin{Example}[The Unit Cotangent Bundle of $S^n$] Let $(M,\xi)$ be
  the unit cotangent bundle $ST^*S^n$ with the standard contact
  structure $\xi$. A Riemannian metric on $S^n$ gives rise to a
  contact form $\alpha$ on $M$ and, for the standard round metric,
  $\alpha$ is a prequantization contact form as in Example
  \ref{ex:pre-quant}, where $B=\Gr^+(2,n+1)$ is the real ``oriented''
  Grassmannian. (The closed geodesics are the oriented great circles
  on $S^n$, i.e., the intersections of $S^n$ with the oriented
  2-planes in $\R^{n+1}$.) It is not hard to see that
  $\chi(B)=2\lfloor (n+1)/2 \rfloor$. Furthermore, since the minimal
  Chern number $N$ of a (simply connected) hypersurface of degree $d$
  in $\CP^{m+1}$ is $m+2-d$ and since $B$ is a complex quadric in
  $\CP^n$, we have $N=n-1$; cf.\ \cite[p.\ 88]{LM} and \cite[Example
  4.27 and Exercise 6.20]{MS}. Thus $\chi^-(\xi)=0$ and
  $\chi^+(\xi)=1/2+ 1/(n-1)$ if $n\geq 3$ is odd and $\chi^+(\xi)=1/2+
  1/2(n-1)$ when $n\geq 2$ is even.  When $n=2$ we have
  $\chi^+(\xi)=1/2$ for $\Gamma=\{0\}$ and $\chi^+(\xi)=1$ for
  $\Gamma=\pi_1(M)=\Z_2$. Alternatively, $\chi^\pm(ST^*S^n)$ can be
  calculated directly via contact homology; see \cite{VK,VK:br}. Note
  that here one can use either the cylindrical contact homology or the
  linearized contact homology with the filling $W$ of $M$ by the unit
  ball bundle in $T^*S^n$. (The only modification that is needed in
  the latter case is that when $n=2$, the fiber $x$ is contractible in
  $W$ and $\chi^+(\xi)=1$ for $\Gamma=\{0\}$.)

\end{Example}

We refer the reader to \cite{CDvK,Es,VK,VK:br} for further examples.  

\subsection{Local formula for the MEC} Let now $x$ be a (simple)
closed orbit of the Reeb flow of $\alpha$, which we assume to be
isolated for all iterations, i.e., all iterated orbits $x^\ka$ are
isolated. Note that these orbits are not required to be
non-degenerate. The Poincar\'e return map $F$ of $x$ is a germ of a
smooth map with an isolated fixed point which we also denote by
$x$. Clearly, the Poincar\'e return map of the iterated orbit $x^\ka$
is just $F^\ka$.

Since the fixed point $x$ of $F$ is isolated for all iterations,
the iterated indices $I_\ka(F,x)$ are defined for all $\ka$ and
the sequence $I_\ka(F,x)$ is periodic with period
$q_{\max}$; see Section \ref{sec:index-germ}.

Set the \emph{mean iterated index} of $x$ to be
$$
\sigma(x)=\frac{1}{q_{\max}}\sum_{\ka=1}^{q_{\max}} I_\ka(F,x)
=\lim_{N\to\infty}\frac{1}{N}\sum_{\ka=1}^{N} I_\ka(F,x).
$$
This is a diffeomorphism (and even homeomorphism) invariant of the flow near the orbit
$x$. When $x$ and all its iterations are non-degenerate, we have
$\sigma(x)=I(F,x)$ when $x$ is even and $\sigma(x)=I(F,x)/2$ when $x$
is odd. (A simple closed Reeb orbit is said to be even/odd depending on
whether the fixed point $x$ is even/odd for $F$.)

Recall that $(M^{2n-1},\xi)$ is a closed contact manifold strongly
fillable by an exact symplectically aspherical manifold 
and that a filling $(W,\omega)$, the homotopy class of the section $\fs$ and the
collection $\Gamma$ are fixed, and hence, in particular, we have the
graded space $\HC_*(\xi)$ and the mean indices of the orbits
unambiguously defined.

\begin{Theorem}
\label{thm:mec}
Assume that the Reeb flow of $\alpha$ on $(M^{2n-1},\xi)$ has only
finitely many simple periodic orbits $x_1,\ldots,x_r$, not necessarily
non-degenerate, in the collection $\Gamma$. Then the conditions (CH)
and \eqref{eq:bound} are satisfied with $l_+=2n-3$ and
$l_-=3$. Furthermore, the limit in \eqref{eq:Euler1} exists for both
the positive and negative MEC of $(M,\xi)$, and
\begin{equation}
\label{eq:mec}
\chi^\pm(\xi)={\sum}^\pm\frac{\sigma(x_i)}{\Delta(x_i)},
\end{equation}
where ${\sum}^\pm$ stands for the sum taken over the orbits $x_i$ with
positive/negative mean index $\Delta(x_i)$.
\end{Theorem}

This result generalizes the non-degenerate case of \eqref{eq:mec} 
proved in \cite{GK} and inspired by \cite{Vi:res} where
such a formula was established for
$(S^{2n-1},\xi_{0})$. A variant of
\eqref{eq:mec} for convex hypersurfaces in $\R^{2n}$ is proved in
\cite{WHL}. The general form of \eqref{eq:mec} given here is
literally identical to a MEC formula obtained in \cite{HM} (see
\cite[Theorem 1.5]{HM}), and probably to the one established in
\cite{LLW} for $(S^{2n-1},\xi_0)$, as can be seen by comparing the
proofs. The main difference lies in the definitions of the numerators:
in \cite{HM}, $\sigma(x_i)$ is defined as the MEC of the local contact
homology of $x_i$, which is a contact invariant, while here
$\sigma(x_i)$ is defined purely topologically as a diffeomorphism and
even homeomorphism invariant of the Poincar\'e return map. (Likewise,
in \cite{LLW}, the numerators are defined via certain local homology
associated with $x_i$.) Finally note that the Morse--Bott version of
\eqref{eq:mec} and its variant for the so-called asymptotically
finite contact manifolds are proved in \cite{Es}.  (The Morse--Bott
version for the geodesic flows is originally established in \cite{Ra}.)

\begin{Remark}
\label{rmk:conditions_mec}
The condition of Theorem \ref{thm:mec} that $\omega$ is exact on $W$
can be slightly relaxed. Namely, when $\Gamma=\{0\}$, it suffices to
assume that $\omega|_{\pi_2(W)}=0$. In general, when
$\Gamma\neq\{0\}$, it is enough to require $\omega$ to be atoroidal;
cf.\ \cite{GHHM}. However, the assumption that $c_1(TW)=0$ or at least
that $c_1(TW)|_{\pi_2(W)}=0$ appears to be essential.
\end{Remark}

\subsection{Preliminaries and the proof} Our
goal in this section is to prove Theorem \ref{thm:mec}. To this end, we
need first to recall several definitions and results concerning filtered
and local contact homology following mainly \cite{HM} and
\cite{GHHM}. Throughout the section, we continue to work in the
setting of Section \ref{sec:setting}. In particular, the contact
manifold $(M,\xi)$ is fixed as are the background
structures used in the definition of the contact homology (the filling $W$,
the section $\fs$ and the collection $\Gamma$). A contact form,
$\alpha$, is always assumed to support $\xi$, i.e., $\ker\alpha=\xi$.

The proof of Theorem \ref{thm:mec} has non-trivial overlaps with the
proof of its counterpart in \cite{HM} -- after all, the difference
between the two results lies mainly in the interpretation of
\eqref{eq:mec} -- although our argument is rather more
concise. Moreover, our proof is not entirely self contained and
depends on \cite{HM} at two essential points. One is the construction
of the local contact homology and the other is one of its
properties. This is (LC3) stated in Section \ref{sec:local}; see also
\cite[Theorem 3 and Corollary~1]{GHHM}. We further elaborate on the
relation between the two theorems and show how, once some preliminary
work is done, our Theorem \ref{rmk:chi} can be derived from
\cite[Theorem 1.5]{HM} in Remark \ref{rmk:chi}.

\subsubsection{Filtered contact homology}
Let us first assume that
the contact form $\alpha$ is non-degenerate. The complex
$\CC_*(\alpha)$ is filtered by the action 
$$
\CA_\alpha(x)=\int_x\alpha, 
$$
i.e., its subspace $\CC_*^a(\alpha)$, where $a\in\R$, generated by the
orbits $x$ with $\CA_\alpha(x)\leq a$ is a subcomplex. We
refer to the homology $\HC_*^a(\alpha)$ of this complex as the
\emph{filtered contact homology}; see, e.g., \cite{GHHM,HM}. Note that,
in contrast with $\HC_*(\xi)$, these spaces depend on
$\alpha$. Whenever $a\leq b$, we have a natural map
$$
\HC_*^a(\alpha)\to \HC_*^b(\alpha)
$$
induced by the inclusions of the complexes and,
since homology commutes with direct limits,
\begin{equation}
\label{eq:limit}
\varinjlim_{a\to\infty} \HC_*^a(\alpha)=\HC_*(\xi).
\end{equation}

This definition extends to degenerate forms by continuity. Below we
will discuss this definition in detail; for the construction is
essential for what follows and it involves certain non-obvious (to us)
nuances. Let us start, however, by recalling some standard facts and
definitions. 

The action spectrum $\CS(\alpha)$ of $\alpha$ is the
collection of action values $A_\alpha(x)$, where $x$ ranges through the
set of closed Reeb orbits of $\alpha$ (in the class $\Gamma$). This is
a closed zero-measure subset of $\R$.

For any two contact forms $\alpha$ and $\alpha'$, giving rise to the
same contact structure $\xi$, write $\alpha'> \alpha$ when
$\alpha'/\alpha > 1$, i.e., $\alpha'=f\alpha$ with $f> 1$. This is
clearly a partial order, and, when both forms are non-degenerate, we
have a homomorphism $\HC_*^a(\alpha')\to \HC_*^a(\alpha)$ induced
by the natural cobordism in the symplectization of $\xi$ between
$\alpha$ and $\alpha'$.
 
Furthermore, we will need the following invariance property of
filtered contact homology, stated in a slightly different form in
\cite[Proposition 5]{GHHM}: \emph{Let $\alpha_s$, $s\in [0,\, 1]$, be
  a family of contact forms such that $\alpha_0$ and $\alpha_1$ are
  non-degenerate and $a\not\in\CS(\alpha_s)$ for all $s$. Then the
  contact homology spaces $\HC^a_*(\alpha_0)$ and $\HC^a_*(\alpha_1)$
  are isomorphic and the isomorphism is independent of the family
  $\alpha_s$ as long as $a\not\in \CS(\alpha_s)$. Furthermore, assume
  that the family $\alpha_s$ is decreasing, i.e.,
  $\alpha_{s'}>\alpha_s$ when $s'<s$. Then the natural map
  $\HC^a_*(\alpha_0)\to \HC^a_*(\alpha_1)$ is an isomorphism.} This
proposition can be proved in exactly the same way as its Floer
homological counterparts; see, e.g., \cite{BPS, Gi:coiso,
  Vi:functors}.

Let now $\alpha$ be a possibly degenerate contact form and
$a\not\in\CS(\alpha)$. Set 
$$
\HC_*^a(\alpha):=\HC_*^a(\tal),
$$
where $\tal$ is a non-degenerate $C^\infty$-small perturbation of
$\alpha$, i.e., $\tal$ is non-degenerate and $C^\infty$-close to
$\alpha$. By the invariance property of the filtered contact homology, for
any two non-degenerate perturbations sufficiently close to $\alpha$
the homology groups on the right hand side are canonically isomorphic,
and hence the left hand side is well defined. Alternatively, we could
have set
$$
\HC_*^a(\alpha):=\varinjlim_{\tal>\alpha}\HC_*^a(\tal),
$$
where the limit is taken over all non-degenerate forms
$\tal>\alpha$. These two definitions are obviously equivalent. In
contrast with $\HC_*(\xi)$, the graded vector spaces $\HC_*^a(\alpha)$
are automatically finite dimensional:
$\sum_l\dim\HC_l^a(\alpha)<\infty$.

With this definition, we have a well-defined map
$\HC_*^a(\alpha')\to \HC_*^a(\alpha)$ for any two, not necessarily
non-degenerate, forms $\alpha'>\alpha$, and hence the non-degeneracy requirement
on $\alpha_0$ and $\alpha_1$ in the invariance property can be
dropped.  Now, however, once $\alpha$ is not assumed to be
non-degenerate, \eqref{eq:limit} requires a proof. Implicitly, this
result is already contained in \cite{HM}. 

\begin{Lemma}
\label{lemma:limit}
For any, not necessarily non-degenerate form $\alpha$,
\eqref{eq:limit} holds as $a\to\infty$ through $\R\setminus \CS(\alpha)$.
\end{Lemma}

\begin{proof}
Fix a sequence $a_i\to\infty$ with $a_i\not\in \CS(\alpha)$ and
consider a decreasing sequence of non-degenerate forms $\alpha_j$ $C^\infty$-converging to
$\alpha$ from above, i.e., we have
$$
\alpha_1>
\alpha_2>\cdots >\alpha .
$$
These sequences give rise to the maps
\begin{equation}
\label{eq:maps_j}
\ldots\to\HC_*^{a_i}(\alpha_j)\to
\HC_*^{a_{i}}(\alpha_{j+1}) \to\ldots.
\end{equation}
In addition, of course, we also have the homomorphisms
\begin{equation}
\label{eq:maps_i}
\ldots\to \HC_*^{a_i}(\alpha_j)\to
\HC_*^{a_{i+1}}(\alpha_j)\to\ldots 
\end{equation}
commuting with the maps \eqref{eq:maps_j}. After passing to the
limit as $i\to\infty$, the maps \eqref{eq:maps_j} induce the identity
map on $\HC_*(\xi)$.

Set 
$$
L=\varinjlim_{a_i\to\infty} \HC_*^{a_i}(\alpha_i)
$$
with respect to the ``diagonal'' maps
$$
\delta_i\colon \HC_*^{a_i}(\alpha_i)\to
\HC_*^{a_{i+1}}(\alpha_{i+1}).
$$
To prove the lemma, it suffices to show that
$L\cong\HC_*(\xi)$.

Let $u\in \HC_*(\xi)$. By \eqref{eq:limit}, there exists $i(u)$ such
that $u$ is the image of
$u_{i(u)1}\in\HC_*^{a_{i(u)}}(\alpha_1)$. Applying the maps
\eqref{eq:maps_j} and \eqref{eq:maps_i} to $u_{i(u)1}$, we obtain a
double sequence $u_{ij}\in \HC_*^{a_i}(\alpha_j)$ where $i\geq
i(u)$. Set $v_i=u_{ii}$, $i\geq i(u)$. Clearly,
$\delta_i(v_i)=v_{i+1}$, and thus the sequence $v_i$ gives rise to an
element $v\in L$. As a result, we have constructed a homomorphism
$$
\Phi\colon \HC_*(\xi)\to L, \quad \Phi(u)=v.
$$

Conversely, consider a sequence $v_i\in\HC_*^{a_i}(\alpha_i)$
with $i\geq i(v)$ such
that $\delta_i(v_i)=v_{i+1}$. The image $u$ of $v_i$ in
$\HC_*(\xi)$ is independent of choice of $v_i$, and we have a map
$$
\Psi\colon L\to \HC_*(\xi), \quad \Psi(v)=u .
$$
Essentially by the definition, $\Psi\Phi=\id$. \end{proof}

Finally let $\HC_*^{(a,\,b)}(\alpha)$, where $a<b$, be the homology of the
quotient complex $\CC_*^b(\alpha)/\CC_*^a(\alpha)$, provided that
$\alpha$ is non-degenerate. When $a<b$ are outside $\CS(\alpha)$ this
definition again extends to all $\alpha$ by continuity. In any event,
we obviously have the exact sequence
$$
\ldots\to \HC_*^a(\alpha)\to \HC_*^b(\alpha)\to
\HC_*^{(a,\,b)}(\alpha)\to\ldots .
$$

\subsubsection{Local contact homology} 
\label{sec:local}
Let $x$ be an isolated, but not necessarily simple or non-degenerate,
closed Reeb orbit of $\alpha$. Consider a small non-degenerate
perturbation $\tal$ of $\alpha$. Under this perturbation $x$ splits
into, in general, several non-degenerate orbits $\tx_1,\ldots,\tx_r$
with nearly the same period (i.e., action) and mean index as $x$. The
vector space $\CC_*(\tal,x)$ generated over $\Q$ by the good orbits in
this collection can be equipped with a differential. The resulting
homology, denoted by $\HC_*(x)$, is independent of the perturbation
$\tal$; see \cite{HM}. We call $\HC_*(x)$ the \emph{local contact
  homology} of $x$. These complexes and the homology spaces are graded
just as the ordinary contact homology, i.e.,
$|\tx_i|=\MUCZ(\tx_i)+n-3$.

\begin{Example} Assume that $x$ is non-degenerate. Then $\HC_*(x)$ is
  equal to $\Q$ and concentrated in degree $|x|$ when $x$ is good and
  $\HC_*(x)=0$ when $x$ is bad.
\end{Example}

\begin{Example} 
  Assume that $x$ is simple. Then $\HC_*(x)$ is isomorphic up to a
  shift of degree to the local Floer homology $\HF_*(F)$ of the
  Poincar\'e return map $F$ of $x$. See \cite{GHHM,HM} for the proof
  of this fact, which can also be established by repeating
  word-for-word the proof of \cite[Proposition 4.30]{EKP}, and, e.g.,
  \cite{Gi:CC,GG:gap,McL} for a detailed discussion of local Floer
  homology.
\end{Example}

\begin{Example}
\label{ex:contact-Floer}
When $x$ is not simple, the relation between the local contact
homology and the Floer homology is more involved. Namely,
$\HC_{*+n-3}(y^\ka)=\HF_*^{\Z_\ka}(F^\ka)$, where $F$ is the
Poincar\'e return map of a simple orbit $y$ and the right hand side
stands for the $\Z_\ka$-equivariant local Floer homology;
\cite[Theorem~3]{GHHM}. The proof of this theorem depends on the
machinery of multivalued perturbations and is only outlined in
\cite{GHHM}. Although we find this relation illuminating, the present
paper does not rely on this result.
\end{Example}

The proof of Theorem \ref{thm:mec} makes use of several properties of
local contact homology, which we now recall following \cite{GHHM,HM}:

\begin{itemize}
\item[(LC1)] $\HC_*(x)$ is finite dimensional and supported in the
  range of degrees $[\Delta(x)-2,\Delta(x)+2n-4]$, i.e., $\HC_*(x)$ vanishes when
  $*$ is outside this range.

\item[(LC2)] Assume that $x=y^\ka$, where $y$ is a simple closed Reeb
  orbit, and let $F$ be the Poincar\'e return map of $y$. Then, in the
  notation of Section \ref{sec:index-germ},
$$
\sum (-1)^l\dim \HC_l(x)=I_\ka(F,y).
$$
\item[(LC3)] For all iterations $\ka$ of a simple orbit $x$, the (total)
  dimension of the graded vector space $\HC_*(x^\ka)$ is bounded by
  a constant independent of $\ka$, provided that $x^\ka$
  remains isolated.

\item[(LC4)] Assume that $c$ is the only point of $\CS(\alpha)$ in the
  interval $[a,\, b]$ and that all closed Reeb orbits with action $c$
  are isolated. (Hence, there are only finitely many such orbits.) Then
$$
\HC_*^{(a,\,b)}(\alpha)=\bigoplus_{\CA_{\alpha}(x)=c}\HC_*(x).
$$
\end{itemize}

Here only (LC3) is not straightforward to prove. The assertion (LC1)
follows from the facts that
$$
|\MUCZ(x)-\Delta(x)|\leq n-1
$$ 
for any closed Reeb orbit on a $(2n-1)$-dimensional contact manifold
and that $\Delta(x)$ depends continuously on $x$; see, e.g.,
\cite{Lo,SZ}. The property (LC2) is an immediate consequence of the
definition, and (LC4) follows from the fact that a holomorphic curve
in the symplectization with zero $\omega$-energy must be the cylinder
over a Reeb orbit; see, e.g., \cite[Lemma 5.4]{BEHWZ}. Finally, (LC3)
is a far reaching generalization a theorem of Gromoll and Meyer,
\cite{GM}. This result is established in \cite[Section 6]{HM} as a
consequence of a theorem from \cite{GG:gap}, asserting a similar upper
bound for local Floer homology. (Note that (LC3) also follows from
\cite[Theorem 3 and Corollary 1]{GHHM} stated in Example
\ref{ex:contact-Floer} above.  An analogue of (LC3) for
(non-equivariant) symplectic homology is proved in \cite{McL}.)

\subsubsection{Proof of Theorem \ref{thm:mec}} In the proof, we will
focus on the case of $\chi^+$; for its negative counterpart, $\chi^-$,
can be handled in a similar fashion. Thus let $x_1,\ldots,x_r$ be the
orbits of $\alpha$ with positive mean index: $\Delta(x_i)>0$. Set
$l_+=2n-3$. By (LC1), this is the lowest degree for which the orbits
with $\Delta=0$ cannot contribute to the homology.

Our first goal is to prove 
\eqref{eq:bound} and (CH). To this end, observe that we have the following
version of the Morse inequalities:
\begin{equation}
\label{eq:morse}
\dim\HC_l(\alpha)\leq \sum\dim \HC_l(x),
\end{equation}
where, when $l\geq l_+$, the sum is taken over all orbits $x=x_i^\ka$
with $\Delta(x)>0$. This is a consequence of (LC4) and the long exact
sequence for filtered contact homology. Since the mean index is
homogeneous, i.e., $\Delta(x)=\ka\Delta(x_i)$, we see that when $\ka$
is large the orbits $x_i^\ka$ do not contribute to $\HC_l(\alpha)$ or,
to be more precise, to the right hand side of \eqref{eq:morse}, due to
(LC1) again. (It suffices to take $\ka> (l+2)/\Delta(x_i)$ here.) In
particular, the right hand side of \eqref{eq:morse} is finite when
$l\geq l_+$.  Furthermore, by (LC3), the contribution of $x_i^\ka$ to
$\HC_l(\alpha)$ is bounded from above and the bound is independent of
$l$. This proves \eqref{eq:bound}: $\dim\HC_l(\alpha)\leq \const$,
where the constant is independent of $l\geq l_+$. 

Let us now prove \eqref{eq:mec} and, particular, the fact that
$\chi^+(\xi)$ is defined. Set
$$
\chi(x)=\sum_l(-1)^l\dim\HC_l(x).
$$
Thus, in the notation of (LC2), $\chi(x) =I_\ka(F)$ where $x=y^\ka$
and $F$ is the Poincar\'e return map of $y$.  Next note that for any
$a\not\in\CA(\alpha)$, we have
\begin{equation}
\label{eq:euler-ch}
\sum_l(-1)^l\dim\HC_l^a(\alpha)= \sum_{\CA_\alpha(x)<a} \chi(x),
\end{equation}
where the sum is now over all closed Reeb orbits with action less than
$a$, but not as in \eqref{eq:morse} over only the iterations of the
orbits $x_i$.  This is again an immediate consequence of (LC2) and
(LC4) and the long exact sequence for filtered contact homology.

When the summation is restricted to the degrees
$l_+\leq l\leq N$,  a variant of \eqref{eq:euler-ch} still holds up to a
bounded error: 
\begin{equation}
\label{eq:euler2}
\Big|\sum_{l=
  l_+}^{N}(-1)^l\dim\HC_l^a(\alpha)-\sum
\chi(x_i^\ka)\Big|
\leq \const,
\end{equation}
where $\const$ on the right is independent of $a$ and $N$ and the
second sum is taken now over the orbits $x_i^\ka$ with
$\CA_\alpha(x_i^\ka)<a$ and $\Delta(x_i^\ka)\leq N$.  Just as
\eqref{eq:morse} and \eqref{eq:euler-ch}, this readily follows from
(LC2) and (LC4) and the long exact sequence for filtered contact
homology.

Let us divide \eqref{eq:euler2} by $N$ and let $a\to\infty$ and then
$N\to\infty$. By Lemma \ref{lemma:limit}, the first sum will then converge
to $\chi^+(\xi)$, and hence
$$
\chi^+(\xi)=\sum_i
\lim_{N\to\infty}\frac{1}{N}\sum_{\Delta(x_i^\ka)\leq N}\chi(x_i^\ka),
$$
provided that the limit on the right hand side exists.
Since $\Delta(x_i^\ka)=\ka\Delta(x_i)$, we also have
$$
\sum_{\Delta(x_i^\ka)\leq
  N}\chi(x_i^\ka)=\frac{1}{\Delta(x_i)}\sum_{\ka=1}^N\chi(x_i^\ka)
+O(1),
$$
as $N\to\infty$. As a result, by the definition of $\sigma(x_i)$ and (LC2),
$$
\lim_{N\to\infty}\frac{1}{N}\sum_{\Delta(x_i^\ka)\leq N}\chi(x_i^\ka)=
\frac{\sigma(x_i)}{\Delta(x_i)},
$$
which concludes the proof of the theorem. $\hfill\square$ 

\begin{Remark}
 To guarantee the existence of the positive/negative MEC
  and prove \eqref{eq:bound} for positive/negative range, it suffices
  to assume only that the collection of simple Reeb orbits with
  positive/negative mean index is finite.
\end{Remark}

\begin{Remark}
\label{rmk:chi}
By (LC2) and Theorem \ref{thm:subordinate}, the function $\ka\mapsto
\chi(x^\ka)$ is subordinated to the collection $\CN(F)$ associated
with the Poincar\'e return map $F$ of $x$ as in Section
\ref{sec:index-germ}. Furthermore, recall that the local MEC of $x$ is
defined in \cite{HM} as
$$
\hat{\chi}(x)=\lim_{N\to\infty}\frac{1}{N}\sum_{\ka=1}^N\chi(x^\ka).
$$
By (LC2), we have $\hat{\chi}(x)=\sigma(x)$. (In particular, the limit
in the definition of $\hat{\chi}(x)$ exists.) With these observations
in mind, our Theorem \ref{thm:mec} can be easily obtained as a
consequence of \cite[Theorem 1.5]{HM}.
\end{Remark}

\subsection{Asymptotic Morse inequalities} 
\label{sec:morse}
The argument from the
previous section lends itself readily to a proof of the asymptotic
Morse inequalities. Namely, in the setting of Section
\ref{sec:setting}, assume that $\xi$ satisfies (HC) and set
$$
\beta^\pm(\xi)=\limsup_{N\to\infty}\frac{1}{N}\sum_{l=l_\pm}^N\dim\HC_{\pm l}(\xi).
$$
Likewise, when $x$ is a simple closed Reeb orbit such that all
iterations $x^\ka$ are isolated, set
$$
\beta (x) =\limsup_{N\to\infty}\frac{1}{N}\sum_{\ka=1}^N\dim\HC_*(x^\ka).
$$
This is finite number by (LC3); in fact, one can expect that in this
case the limit exists; cf.\ Remark~\ref{rmk:homology}.

\begin{Theorem}
\label{thm:morse}
Assume that the Reeb flow of $\alpha$ on $M^{2n-1}$ has only
finitely many simple periodic orbits $x_1,\ldots,x_r$, not necessarily
non-degenerate, in the collection $\Gamma$. Then $\beta^\pm(\xi)$ is
finite and moreover
$$
\beta^\pm(\xi)\leq {\sum}^\pm\frac{\beta (x_i)}{\Delta(x_i)},
$$
where ${\sum}^\pm$ stands for the sum taken over the orbits $x_i$ with
positive/negative mean index $\Delta(x_i)$.
\end{Theorem}

This theorem can be proved exactly in the same way as Theorem
\ref{thm:mec}. The non-degenerate case of this result is pointed out
in \cite[Remark 1.10]{GK}. Note that, as an immediate consequence of
Theorem \ref{thm:morse}, we obtain a theorem from \cite{HM} (cf.\
\cite{GM,McL}) asserting that $\alpha$ must have infinitely many closed
Reeb orbits when $\dim\HC_l(\xi)$ is unbounded as a function of $\pm
l\geq l_\pm$.

\begin{Remark}
\label{rmk:homology}
  At this stage, very little is known about the sequence
  $\dim\HC_*(x^\ka)$, except that this sequence is bounded,
  or about the ``mean Betti number'' $\beta(x)$. For instance, recall
  that the sequence $\chi(x^\ka)=I_\ka(F)$ is periodic by Theorem
  \ref{thm:subordinate} and (LC2), and hence $\sigma(x)$ is
  rational. Moreover, a similar result holds for the local Floer homology. In
  fact, as is easy to see from the proof of \cite[Theorem
  1.1]{GG:gap}, the sequence $\dim\HF_*(x^\ka)$, where $x$ is an
  isolated periodic orbit of a Hamiltonian diffeomorphism, is still
  subordinated to $\CN$.  We conjecture that this is also true for
  $\dim\HC_*(x^\ka)$, and hence $\beta(x)\in\Q$. However, neither of
  these facts have been proved yet.
\end{Remark}

\begin{Remark}[Variations]
\label{rmk:var}
Variants of Theorems \ref{thm:mec} and \ref{thm:morse} hold for
cylindrical contact homology, when the latter is defined, with
straightforward modifications to the setting and virtually the same
proofs; cf.\ Remark \ref{rmk:cyl}.

One can also expect the analogues of these theorems to hold for, say,
the (positive) equivariant symplectic homology or for the
$S^1$-equivariant Morse homology of the energy functional on the space
of loops on Riemannian or Finsler manifold, or for the equivariant
Morse-type homology associated with certain other functionals on loop
spaces; see \cite{CDvK,FSVK,LLW,Ra,WHL} for some relevant results. However,
now the situation is more delicate because the corresponding complexes
are not directly, by definition, generated by the (good) orbits and,
strictly speaking, even in the non-degenerate case an extra step in
the proof  is needed to relate the homology and the orbits. For
instance, in the non-degenerate analogue of Theorem \ref{thm:mec} for
either of these ``homology theories'', the contribution of an orbit
$x^\ka$ is equal to the Euler characteristic of the infinite lens
space $S^{2\infty-1}/\Z_\ka$ with respect to a certain local
coefficient system; cf.\ \cite[Lecture 3]{Bott}. This contribution is
trivial when $x^\ka$ is bad and equal to $\pm 1$ when the orbit is
good. With this in mind, the proof of Theorem \ref{thm:mec} should go
through with mainly notational changes except for an analogue of (LC3),
which now requires a proof; cf.\ \cite{HM} and \cite{McL}.
\end{Remark}

\section{Closed Reeb orbits on $S^3$}
\label{sec:S^3}

Our goal in this section is to reprove the following 

\begin{Theorem}[\cite{CGH,GHHM,LL}]
\label{thm:S^3}
Let $\alpha$ be a contact form on $S^3$ such that $\ker\alpha$ is the
standard contact structure. Then the Reeb flow of $\alpha$ has at
least two closed orbits.
\end{Theorem}

In fact, the result proved in \cite{CGH} is much more general and
holds for all closed contact 3-manifolds; the proof uses the machinery
of embedded contact homology. The proofs in \cite{GHHM,LL} both rely
on a theorem establishing the ``degenerate case of the Conley
conjecture'' for contact forms and more specifically asserting that
the presence of one closed Reeb orbit of a particular type (the
so-called SDM) implies the existence of infinitely many closed Reeb
orbits.  This result holds in all dimensions and is stated explicitly
below. Another non-trivial, and in this case strictly
three-dimensional, counterpart of the argument in \cite{GHHM} comes
from the theory of finite energy foliations (see \cite{HWZ1,HWZ2}),
while the argument in \cite{LL} uses, also in a non-trivial way, a
variant of the resonance relation formula from \cite{LLW}. Here, we
give a very simple proof of Theorem \ref{thm:S^3} establishing it as a
consequence of our Theorem \ref{thm:mec} and the ``Conley conjecture''
type result mentioned above.

To state this result, we need to recall some definitions. Let $F_t$,
$t\in [0,\,1]$, be a family of germs of Hamiltonian diffeomorphisms of
$\R^{2n}$ fixing the origin $0\in\R^{2n}$ and generated by the germs
of Hamiltonians $H_t$ depending on $t\in S^1$. Set $F=F_1$ and
let us assume that the origin is an isolated fixed point of $F$.

Denote by $\Delta(F)\in\R$ and $\HF_*(F)$ the mean index and,
respectively, the local Floer homology of $F$ over $\Q$. Again, we refer the
reader to \cite{Gi:CC,GG:gap,McL} for the definition of the latter.
Here we only note that both of these invariants depend actually on the
homotopy type (with fixed end points) of the entire family
$F_t$. However, changing this family to a non-homotopic one, results
in a shift of $\Delta(F)$ and of the grading of $\HF_*(F)$ by the same
even number. Furthermore, $\HF_*(F)$ is supported in the range
$[\Delta(F)-n,\Delta(F)+n]$, i.e., $\HF_*(F)$ vanishes for degrees
outside this range. When $2n=2$, $\HF_*(F)$ is supported in at most
only one degree; \cite{GG:gap}.  Finally, analogously to (LC2) and
essentially by definition, we have
$$
I(F)=(-1)^n\sum (-1)^l\dim\HF_l(F).
$$

The fixed point of $F$, the origin, is said to be a
\emph{symplectically degenerate maximum} (SDM) of $F$ if $\Delta(x)\in
\Z$ and $\HF_*(F)$ does not vanish at the upper boundary of its
possible range, i.e., $\HF_{\Delta(x)+n}(F)\neq 0$. This is clearly a
feature of the time-one map $F$, independent of the homotopy type of
the family $F_t$. Note also that in this case $F$ is necessarily
totally degenerate, i.e., all eigenvalues of $DF$ are equal to 1,
although $DF$ need not be the identity. (Clearly, $\Delta(F)\in\Z$
when $F$ is totally degenerate; the converse, however, is not true
without homological requirements.) Furthermore, then
$\HF_{\Delta(x)+n}(F)=\Q$ and $\HF_*(F)=0$ in all other degrees.  (See
\cite{Gi:CC,GG:gap} for this and other results and a detailed
discussion of SDM's.)

\begin{Example}
  Assume that $H$ on $\R^{2n}$ is autonomous, has an isolated critical
  point at the origin and this critical point is a local maximum, and
  that the Hessian $d^2 H$ at the origin is identically zero. (More
  generally, it suffices to assume that the eigenvalues of $d^2 H$ with
  respect to the symplectic form are all zero.) Then the origin is an
  SDM of the time-one map $F$ generated by $H$.
\end{Example}

Returning to Reeb flows, we say that an isolated closed Reeb orbit $x$
is an SDM if the corresponding fixed point of the
Poincar\'e return map $F$ of $x$ is an SDM. (This is equivalent to requiring
that $\Delta(x)\in\Z$ and $\HC_l(x)$ is $\Q$ when $l=\Delta(x)+2n-4$
and zero otherwise; see \cite{GHHM} for this and other related
results.) Note that, although the grading of $\HF_*(x)$ depends on some
extra choices, e.g., a trivialization of $\xi$ along $x$, the notion
of SDM is independent of these choices. 

In the setting of Section \ref{sec:setting}, we have

\begin{Theorem}[\cite{GHHM}]
\label{thm:sdm}
Let $(M^{2n-1},\ker\alpha)$ be a contact manifold admitting a strong
exact symplectic filling $(W,\omega)$ such that $c_1(TW)=0$. Assume
that the Reeb flow of $\alpha$ has an isolated simple closed Reeb
orbit $x$ which is an SDM. Then the Reeb flow of $\alpha$ has
infinitely many periodic orbits.
\end{Theorem}

\begin{Remark} The proof of Theorem \ref{thm:sdm} is a
  straightforward, although lengthy and cumbersome, adaptation of the
  proof of the degenerate case of the Conley conjecture in
  \cite{GG:gaps}.  In fact, in \cite{GHHM}, the theorem is established
  under somewhat less restrictive conditions. Namely, it suffices to
  require that $\omega|_{\pi_2(W)}=0$, when $x$ is contractible
  (rather than that $\omega$ is exact) and that $\omega$ is atoroidal
  in general; cf.\ Remark \ref{rmk:conditions_mec}.
\end{Remark}

Replacing the upper boundary of the range by the lower boundary in the
definition of an SDM, we arrive at the notion of a
\emph{symplectically degenerate minimum} (SDMin). Thus $F$ has an
SDMin at the origin if and only if $\Delta(F)\in\Z$ and
$\HF_{\Delta(F)-n}(F)\neq 0$. This notion naturally arises in dynamics
(see \cite{He:cot} and also \cite[Remark 1]{GHHM}), and Theorem
\ref{thm:sdm} holds with an SDM replaced by an SDMin.

Just as the local Floer homology in dimension two, the local contact
homology of an isolated closed Reeb orbit $x$ on a 3-manifold is
concentrated in at most one degree. Hence, in this case we have the
following mutually exclusive possibilities when $x$ is isolated and
degenerate, but not necessarily simple:
\begin{itemize}
\item SDM: $\HC_{\Delta(x)}(x)\neq 0$ and $x$ is an SDM,
\item SDMin: $\HC_{\Delta(x)-2}(x)\neq 0$ and $x$ is an SDMin,
\item Saddle or ``Monkey Saddle'': $\HC_{\Delta(x)-1}(x)\neq 0$,
\item Homologically Trivial: $\HC_*(x)=0$.
\end{itemize}
All of these cases do occur. Furthermore, for any isolated iterated
closed Reeb orbit $y^\ka$ (not necessarily non-degenerate), we have
\begin{equation}
\label{eq:dim}
(-1)^l\dim \HC_l(y^\ka)=I_\ka(F),
\end{equation}
where $F$ is the Poincar\'e return map of $y$ and $l$ is the degree where
the cohomology of $y^\ka$ is supported. (Note that $\HC_*(y^\ka)=0$ if
and only if $I_\ka(F)=0$.) This is an immediate consequence of (LC2)
and again the fact that the cohomology is supported in only one degree.

\begin{proof}[Proof of Theorem \ref{thm:S^3}]
Arguing by contradiction, assume that the Reeb flow of $\alpha$ on
$(S^3,\xi_0)$ has only one closed
simple orbit $x$. We will prove that then $x$ is an SDM, and hence the
flow has infinitely many periodic orbits by Theorem
\ref{thm:sdm}. Applying Theorem \ref{thm:mec} and Example \ref{eq:xi_0}, we have
\begin{equation}
\label{eq:one-orbit}
\frac{\sigma(x)}{\Delta(x)}=\frac{1}{2},
\end{equation}
where $\Delta(x)>0$. Therefore, $\sigma(x)>0$. 

Furthermore, recall that $\HC_l(\xi_0)=\Q$ when $l\geq 2$ is even and
$\HC_l(\xi_0)=0$ otherwise; see, e.g., \cite{Bo}.

The orbit $x$ is necessarily elliptic: all Floquet multipliers, i.e.,
eigenvalues of $DF$, lie on the unit circle. Indeed, if $x$ is
hyperbolic and even, we have $\sigma(x)=-1$, which is
impossible. If $x$ is hyperbolic and odd, we have
$\sigma(x)=1/2$ and $\Delta(x)=1$. Furthermore, all orbits $x^\ka$ are
non-degenerate and $\MUCZ(x^\ka)=\ka$.  Since the even iterations of
$x$ are bad orbits, the complex $\CC_*(\alpha)$ is generated by the
orbits $x^\ka$, where $\ka$ is odd, with $|x^\ka|=\ka-1$, which is
also impossible since then $\HC_0(\xi_0)=\Q$.

With $x$ elliptic, write $\Delta(x)=2m+2\alpha$, where $m$ is a
non-negative integer and $\alpha\in [0,\,1)$. By the definition of the
mean index, $e^{\pm 2\pi i\alpha}$ is an eigenvalue of $F$. From
\eqref{eq:one-orbit} we see that $\alpha$ is necessarily rational. Set
$\alpha=p/q$, where $q\geq 2$ and $q>p\geq 0$, and $p$ and $q$ are
mutually prime. (In particular, $x^q$ is degenerate.) To summarize, we
are in the setting of Example \ref{ex:2D}: $I_\ka(F)=1-r$ for some
$r\geq 0$ when $q|\ka$, and $I_\ka(F)=1$ otherwise. We conclude that
$$
\sigma(x)=1-\frac{r}{q}\text{ and } \Delta(x)=2m+\frac{2p}{q},
$$
and \eqref{eq:one-orbit} amounts to
$$
1-\frac{r}{q}=m+\frac{p}{q},
$$
where $m$, $r$ and $p$ are non-negative and $q\geq 2$. 
This condition can only be satisfied when $m=0$ or $m=1$.

Let us first examine the case $m=0$, which is exactly the point where
in \cite{GHHM} we had to rely on the results from
\cite{HWZ1,HWZ2}. Then $p>0$ (for otherwise $\Delta(x)=0$) and
$p+r=q$. In particular, $r<q$.

To rule out this case, observe first that $p=1$. Indeed, we have
$|x^\ka|=0$ as long as $\ka p/q <1$. If $p\geq 2$, the first degree
jump occurs when $\ka p /q$ becomes greater than 1 and the degree
increases to 2. The next jumps occur when $\ka p/q $ moves over the
subsequent integers, while $x^\ka$ is still nondegenerate, or when
$\ka=q$. In the latter case, $x^q$ is degenerate, but
$\Delta(x^q)\geq 4$ since $\ka\geq 2$ and $p\geq 2$, and $x^q$ cannot
contribute to the homology in degree $0$. We conclude that again
$\HC_{0}(\alpha)\neq 0$ when $p>1$, which is impossible.

Thus $p=1$. The degrees of $x^\ka$ form a sequence
$$
|x^\ka|=\underbrace{0,\ldots,0}_{q-1},\,*,\,
\underbrace{2,\,\ldots,2}_{q-1},\,*,\ldots ,
$$
where the undefined degrees of the degenerate iterations $|x^\ka|$ are
entered as the asterisks. It is clear, however, that $\HC_*(x^q)$ must be
concentrated in degree one, to cancel the contribution of previous
iterations to degree zero, and hence $r\geq 1$. 

It follows from \eqref{eq:dim} that $\HC_*(x^\ka)=\Q^{r-1}$, concentrated in degree
$2\ka/q-1$, when $q|\ka$ and that $\HC_*(x^\ka)=\Q$,
concentrated in degree $2\lfloor \ka/q \rfloor$, otherwise, i.e.,
when $x^\ka$ is non-degenerate. Hence, all iterates $x^\ka$ with
$\ka\leq q-1$ contribute to degree $0$ while the iterates $x^\ka$
with $\ka>q$ contribute to the degrees greater than or equal to 2:
$$
\HC_*(x^\ka)=\underbrace{\Q_0,\ldots,\Q_0}_{q-1},\,
\Q^{r-1}_1,\,
\underbrace{\Q_2,\,\ldots,\Q_2}_{q-1},\,\Q^{r-1}_3,\ldots ,
$$
where the subscripts on the right hand side indicate the degree.  To
ensure that $\HC_0(\alpha)=0$, we must have $q-1\leq r-1$, by the long
exact sequence for filtered contact homology, i.e., $q\leq r$,
which is impossible since $r<q$ as stated above.

It follows now that $m=1$ and $p=0=r$. Therefore, the orbit $x$ itself is
degenerate and $\Delta(x)=2$. Clearly, since $\HC_*(x)$ is
concentrated in at most one degree, the fact that $\HC_2(\xi_0)=\Q$
implies that $\HC_2(x)=\Q$, and hence $x$ is an SDM. This concludes the
proof of the theorem.
\end{proof}

\begin{Remark}
\label{rmk:symmetry}
Let $\alpha$ be a contact form on $(S^3,\xi_0)$ invariant under a
contact involution $\tau$. (For instance, this is the case when
$\alpha$ comes from a symmetric star-shapped embedding of $S^3$ into
$\R^4$ and $\tau$ is the antipodal involution or when $\alpha$ is the
lift of a contact form on $\RP^3$ supporting the standard contact
structure.)  Then the Reeb flow of $\alpha$ cannot have exactly two
simple Reeb orbits ``swapped'' by $\tau$. Indeed, assume the contrary
and denote the orbits by $x$ and $\tau(x)$. Then these orbits have the
same local invariants: the same mean index, the same mean iterated
index and the same local contact homology. Now, arguing as in the
proof of Theorem \ref{thm:S^3}, but without relying on Theorem
\ref{thm:sdm}, it is not hard to show that this is
impossible. However, as the example of an irrational ellipsoid in
$\R^4$ shows, $\alpha$ can have exactly two simple orbits each of
which is invariant under $\tau$. Projecting the irrational ellipsoid
contact form to $\RP^3=ST^*S^2$, we obtain an asymmetric Finsler
metric on $S^2$ with exactly two geodesics; see \cite{Ka,Zi} and also
\cite{Gu}.
\end{Remark}

As an easy application of Theorem \ref{thm:S^3}, we have

\begin{Corollary} 
\label{cor:finsler}
Let $M$ be a fiberwise star-shaped (with respect to
  the zero section) hypersurface in $T^*S^2$. Then $M$ carries at
  least two closed characteristics. 
\end{Corollary}

When $M$ is fiberwise convex, the Reeb flow is just the geodesic flow
of a (not necessarily symmetric) Finsler metric on $S^2$. Thus, in
particular, it follows that any Finsler metric on $S^2$ has at least
two simple closed geodesics. This result is originally proved in
\cite{BL}. Of course, Corollary \ref{cor:finsler} also immediately
follows from the main theorem of \cite{CGH}.

\begin{proof}
  The hypersurface $M$ is of contact type and contactomorphic to
  $\RP^3=ST^* S^2$ with the standard contact structure $\xi_0$ and some contact
  form $\alpha$. Arguing by contradiction, assume that $(M,\alpha)$
  carries only one simple closed Reeb orbit $x$. We claim that $x$ is not
  contractible.

  This is a standard fact which can be established in a variety of
  ways. For instance, this is a consequence of Remark
  \ref{rmk:symmetry}. Alternatively, this follows from the fact that
  the Reeb flow of any such form $\alpha$ must have non-contractible
  orbits because the cylindrical contact homology of $ST^* S^2$ is
  non-zero for the non-trivial free homotopy type of loops in
  $\RP^3$. The latter assertion can be easily proved by examining the
  indices of the closed orbits for the form coming from an irrational
  ellipsoid in $\R^4$ or by using the Morse--Bott techniques; see,
  e.g., \cite{Bo,VK,VK:br}.  (For Finsler metrics on $S^2$ this is, of
  course, a well-known result in the standard calculus of variations:
  the orbit in question is the energy minimizer within the fixed
  non-trivial free homotopy class; cf.\ \cite{Bott}.)

Finally, since $x$ is not contractible, it lifts to one orbit on $S^3$
and we obtain a contact form on $S^3$ with only one closed Reeb orbit,
which is impossible by Theorem~\ref{thm:S^3}.
 \end{proof}

\end{document}